\documentclass{amsart}
\usepackage[utf8]{inputenc}
\usepackage{amsmath}
\usepackage{amssymb}
\usepackage{stmaryrd}
\usepackage{algorithm}
\usepackage{xcolor}
\usepackage{algorithmicx}
\usepackage{algpseudocode}
\usepackage{graphicx}%
\usepackage{amsfonts}%
\usepackage[title]{appendix}%
\usepackage{mathptmx} 
\usepackage{listings}%
\usepackage{hyperref}

\newtheorem{thm}{Theorem}
\newtheorem{lem}{Lemma}

\newtheorem{rem}{Remark}
\newtheorem{algo}{Algorithm}
\newtheorem{prob}{Problem}

\newcommand{\Div}[0]{\mathrm{div}\,}
\newcommand{\DIV}[0]{\mathbf{div}\,}
\newcommand{\bs}[1]{\boldsymbol{#1}}

\newcommand{\sig}[0]{\bs{\sigma}}

\newcommand{\D}[0]{\bs{D}}
\newcommand{\A}[0]{\bs{A}}
\newcommand{\U}[0]{\bs{u}}
\newcommand{\V}[0]{\bs{v}}
\newcommand{\W}[0]{\bs{w}}
\newcommand{\Lam}[0]{\bs{\lambda}}
\newcommand{\Mu}[0]{\bs{\mu}}
\newcommand{\X}[0]{\bs{\xi}}

\newcommand{\Af}[0]{a}
\newcommand{\Bf}[0]{\mathcal{B}}
\newcommand{\Lf}[0]{\mathcal{L}}
\newcommand{\Sf}[0]{\mathcal{S}_h}
\newcommand{\hf}[0]{h_{\partial \Omega}}
\newcommand{\N}[0]{\bs{n}}

\newcommand{\F}[0]{\bs{f}}
\newcommand{\Z}[0]{\bs{0}}
\newcommand{\Th}[0]{\mathcal{T}_h}
\newcommand{\Eh}[0]{\mathcal{E}_h}
\newcommand{\Fh}[0]{\mathcal{F}_h}
\newcommand{\Gh}[0]{\mathcal{G}_h}

\newcommand{\enorm}[1]{{\left\vert\kern-0.25ex\left\vert\kern-0.25ex\left\vert #1 
            \right\vert\kern-0.25ex\right\vert\kern-0.25ex\right\vert}}

\begin{document}

\author{Tom~Gustafsson \and Juha~Videman}



\title{Stabilised finite element method for Stokes problem with nonlinear slip condition}

\maketitle

\begin{abstract}
    This work introduces a stabilised finite element formulation for the Stokes flow
    problem with a nonlinear slip boundary
    condition of  friction type.
    The boundary condition
    is enforced with the help
    of an additional Lagrange multiplier representing boundary traction
    and the stabilised formulation
    is based
    on simultaneously stabilising
    both the pressure
    and the traction.
    We establish the stability
    and the a priori error analyses,
    and perform a numerical
    convergence study
    in order to verify the theory.
\end{abstract}

\section{Introduction}

The Stokes problem is a well known and extensively studied linear model for creeping flow.
There exist various physically justified
and mathematically valid boundary conditions
that can be directly applied.
For instance, some components of
the velocity field $\U$ may be prescribed
while the other components
are free to vary subject to a zero stress condition.
In the classical linear slip
boundary condition,
the normal velocity
is equal to zero and the tangential velocity remains unspecified.

However, in some cases a more
involved nonlinear interaction takes place between the fluid and
its surroundings.
Think, e.g., of a membrane
leaking only if the pressure becomes
large enough.
Another example is a slip flow
in which the tangential velocity at the boundary becomes nonzero if and only if
the shear stress exceeds
a prescribed, or velocity-dependent, friction threshold.
This inequality-type of friction laws
have been suggested, e.g.,
for the flow of
glaciers over the bedrock~\cite{schoof2006variational}.
Besides, given that the Stokes flow
problem is analogous to the equations
of linear elasticity in the incompressible limit~\cite{kouhia1995linear},
nonlinear slip conditions
could be used for modeling frictional contact
of incompressible solids.
See also Rao--Rajagopal~\cite{rao1999effect}
for more examples on the use of slip boundary conditions
in flow problems.

In this work, we focus on the nonlinear slip condition
with a prescribed Tresca friction threshold.
The nonlinear slip (and leak)
boundary conditions of friction type were first considered for incompressible fluids  by Fujita \cite{Fujita1994}---see also the discussion on slip boundary conditions for fluid flow problems in Le Roux \cite{LeRoux2005}.
These problems can be  written as  variational inequalities of the second kind, cf.~Fujita \cite{Fujita1994}, or, alternatively, as mixed variational inequalities by 
  expressing the boundary traction
as a Lagrange multiplier $\Lam$ which enforces the inequality constraint, cf.~\cite{ayadi2010mixed}.
Here, we adopt the second formulation and propose a stabilised finite element method for its numerical approximation.

Regarding the numerical approximations, Kashiwabara \cite{Kashiwabara2013a} has proven error estimates for the velocity-pressure pair $(\U,p)$
using Taylor--Hood $P_2$--$P_1$ finite elements
but with the inf-sup constant for the tangential component $\Lam_t$,
approximated using the trace space of $P_2$,
still depending on $h$.
The lack of uniform stability  means that it is not
possible to achieve optimal error estimates for $\Lam_t$.
However,  the value of $\Lam_t$ is needed in finding the active constraints at
each iteration step of the solution algorithm.
Therefore, it is  reasonable to aim at  uniform
stability for the  three unknowns $(\U,p,\Lam)$.

Achieving uniform stability simultaneously for $\Lam$ and $p$
can be done by different means, e.g., by a specific
choice of finite element spaces.
The work of Ayadi~et~al.~\cite{ayadi2010mixed,Ayadi2014,Ayadi2019}
is based on the use of $P_1$~bubble--$P_1$--$P_1$ triplet
as a stable choice of mixed finite element spaces.
This is a reasonable choice since $P_1$~bubble--$P_1$
for $(\U,p)$
is known to be stable in the case of the standard Stokes problem
and $P_1$--$P_1$ element for $(\U,\Lam)$ has been implemented to impose
boundary conditions using Lagrange multipliers---although
we would expect that minor modifications of the basis functions
are needed in the case of mixed boundary conditions~\cite{gustafsson2023mortaring}.
Other works based on mixed methods, with or without an explicit Lagrange multiplier
for the boundary condition,
include Djoko et al.~\cite{djoko2016numerical,Djoko2022} and Fang et al.~\cite{Fang2019}.

The present work focuses on residual stabilisation~\cite{barbosa1991finite,stenberg2015error}, i.e.~the inclusion
of additional residual terms in the variational formulation
to circumvent the Babu\v{s}ka--Brezzi condition.
If these residual terms are consistent and scaled properly,
it is possible to have stability for the $(p, \Lam)$ variables,
no matter which finite element spaces are considered
for discretisation.  This will greatly
improve the flexibility in choosing the finite element spaces
and allows for discretisations beyond those based on the boundary conditions.
In this work, the stabilisation allows us to use the lowest
order $P_1$--$P_1$--$P_0$ element in our numerical
experiment, a triplet which would otherwise be
unstable.

Residual stabilisation has been considered in Djoko--Koko~\cite{djoko2022gls}
but only for the velocity--pressure pair and, hence, without  estimates for $\Lam$.
A stabilisation technique through pressure projection has been presented in Li--Li~\cite{li2011pressure}, and  discussed in Qiu et al.~\cite{qiu2018low} and Li et al.~\cite{li2019priori},
but again without estimates for $\Lam$.
Using residual stabilisation for both Lagrange multipliers, $p$ and $\Lam$,
we establish here a uniform stability estimate for all three variables.
This is shown to lead to a quasi-optimality result,
i.e.~a best approximation
result with an additional term due to the unknown location of the slip
boundary,
which is further
refined into an a priori error estimate for the lowest order method.

The work is organized as follows. In Section~2, we present the strong formulation of the problem and in Section~3
derive the corresponding weak formulation.
The stabilised finite element method is presented and
its stability analysed in Section~4.
In Section~5, the quasi-optimality estimate is proven
and shown to provide an optimal  priori error estimate
for the lowest order method.
In Section~6, we derive a solution algorithm
for the discrete variational inequality and,
in Section~7,  report on the results of our numerical
experiment which aims at corroborating the theoretical
convergence rates.

\section{Strong formulation}

Let $\Omega\in \mathbb{R}^d, d\in \{2,3\}$ denote a polygonal (polyhedral)  domain with a Lipschitz boundary $\partial \Omega$ and
let $\U : \Omega \rightarrow \mathbb{R}^d$ be the fluid velocity
field. Denoting the symmetric part of the velocity gradient by
\[
    \D(\U) = \frac12(\nabla \U + \nabla \U^T),
\]
we introduce the differential operator
\[
  \A \U = \U - \,\DIV (2\mu \D(\U)),
\]
where $\mu > 0$ is the kinematic viscosity.
Letting  $p : \Omega \rightarrow \mathbb{R}$ be the pressure field, the balance of linear momentum for an incompressible, homogeneous and linearly viscous fluid reads as 
\begin{equation}
\label{eq:stokes1strong}
\A \U + \nabla p = \F,
\end{equation}
where $\F: \Omega \rightarrow \mathbb{R}^d$ denotes the resultant of external forces. Equation \eqref{eq:stokes1strong}
holds together with the incompressibility constraint
\begin{equation}
\label{eq:stokes2strong}
\Div \U = 0.
\end{equation}

\begin{rem}
    We are considering here the "generalized" Stokes system
    \eqref{eq:stokes1strong}--\eqref{eq:stokes2strong}  for expediency.
    In particular, using the operator $\A\U$, instead of $- \,\DIV (2\mu \D(\U))$,  allows us to impose
    the slip  boundary condition
    on the entire $\partial \Omega$.
   Considering different boundary conditions
    at different parts of the domain requires
    resorting to the trace space
    $H^{1/2}_{00}$ \cite{tartar2007introduction} for the normal components of
    the velocity field trace which, in our opinion, is an unnecessary technical difficulty.
    We also note that
    the generalized equation is relevant, as such, for
    the implicit time discretization
    of the time-dependent Stokes equations,
    and that the solvers presented in this work and available at \cite{gustafsson_tom_2023_8296578}
    can be applied to the
    standard Stokes system by simply
    removing the additional term.
\end{rem}


Next, let us introduce
the boundary conditions. Denoting the Cauchy stress tensor by
\begin{equation}
\label{eq:cauchystress}
\sig(\U,p) = -p \bs{I} +  2\mu \D(\U),
\end{equation}
and  the normal and tangential  components of $\U$ as 
$u_n= \U\cdot \N$ and $\U_t=\U-u_n\N$, where $\N:\partial\Omega\rightarrow \mathbb{R}^d$ is the outward unit normal to $\Omega$, we divide the normal stress vector $\sig(\U,p) \N$ into its normal and tangential components $\sigma_n$ and $\sig_t$ defined through
\[
 \sigma_n(\U,p)=\sig(\U,p)\N\cdot \N
\]
 and 
\[
\sig_t(\U,p)=\sig(\U,p)\N-\sigma_n(\U,p)\N.
\]
On the boundary $\partial \Omega$, we impose the following (nonlinear) slip boundary condition 
\begin{equation}
u_n=0,\quad |\sig_t| \leq \kappa, \quad \sig_t\cdot \U_t+\kappa |\U_t|=0,
\label{slipcond}
\end{equation}
where $\kappa : \partial \Omega \rightarrow (0,\infty)$ is a positive threshold function denoting an upper limit for the tangential stress before slip occurs. In other words, if $|\sig_t|< \kappa $ then the tangential velocity is zero and if $|\sig_t|= \kappa $ then
the tangential stress and velocity vectors are collinear with opposite directions. 
In case the boundary condition is imposed with the help
of Lagrange multipliers,
the definition
\begin{equation}
   \Lam = \sig(\U, p)\N
   \label{lagmult}
\end{equation}
implies that
\begin{equation}
\label{eq:strongtresca}
\lambda_n u_n=0, \quad |\Lam_t| \leq \kappa, \quad \Lam_t \cdot \U_t+\kappa |\U_t|=0,
\end{equation}
where  $\lambda_n = \Lam \cdot \N$ denotes the normal component of the Lagrange multiplier and
 $\Lam_t = \Lam - \lambda_n \N$ its tangential component.

\begin{rem}
    \label{incompressible}
 Note that there exists an analogous
interpretation of the above problem
in solid mechanics.
It is well known that the Stokes
problem \eqref{eq:stokes1strong}--\eqref{eq:stokes2strong} can be
obtained from the equations of linear
elasticity by
defining "pressure" as
the product of the first Lam\'{e} parameter and the divergence of the displacement
field and then
letting the first
Lam\'{e} parameter approach infinity, corresponding to the incompressible limit.
Consequently, in solid mechanics,
the condition \eqref{slipcond}
can be referred to as the Tresca friction condition~\cite{gustafsson2022stabilized}.
\end{rem}


\section{Mixed variational formulation}

We will now present a mixed variational formulation for problem \eqref{eq:stokes1strong},\eqref{eq:stokes2strong}, \eqref{slipcond}. We use the following notation for the  $L^2$ inner products 
\[
(a, b) = \int_\Omega ab \,\mathrm{d}x, \quad (\bs{a}, \bs{b}) = \int_\Omega \bs{a} \cdot \bs{b}\,\mathrm{d}x,
\]
define  the velocity and pressure spaces  through
\[
    \bs{V} = H^1(\Omega)^d, \quad Q = \{ q \in L^2(\Omega) : (q, 1) = 0 \},
\]
and denote the trace space by $\bs{W}=H^{1/2}(\partial \Omega)^d$,
$d\in\{2,3\}$.
The space for the Lagrange multiplier, defined in \eqref{lagmult}, is
\[
\bs{\varLambda} = \{ \Mu \in \bs{M} : -\langle \Mu_t, \bs{v} \rangle \leq (\kappa, |\bs{v}_t|)_{\partial \Omega}~\forall \bs{v} \in \bs{W} \},
\]
where $\bs{M} = \bs{W}^\prime$
and $\langle .,. \rangle$ denotes
the duality pairing between $\bs{M}$ and $\bs{W}$.
In particular,
for $\Mu \in \bs{M}$ and $\bs{v} = (v_n \bs{n} + \bs{v}_t) \in \bs{W}$
we can write
\[
\langle \Mu, \bs{v} \rangle = \langle \mu_n, v_n \rangle + \langle \Mu_t, \bs{v}_t \rangle,
\]
where $\langle \mu_n, v_n \rangle = \langle \mu_n \bs{n}, v_n \bs{n} \rangle$; cf.~\cite{W11}.

The mixed variational formulation of problem \eqref{eq:stokes1strong},\eqref{eq:stokes2strong}, \eqref{slipcond} now reads as follows:
find $(\U, p, \Lam) \in \bs{V} \times Q \times \bs{\varLambda}$ such that
\begin{equation}
\label{eq:weak1}
\left\{
\begin{alignedat}{2}
    a(\U, \V) - (\Div \V, p) - \langle \Lam, \V \rangle &= (\F, \V) \quad &&\forall \V \in \bs{V}, \\
    (\Div \U, q) &= 0 \quad &&\forall q \in Q, \ \\
    -\langle \Mu - \Lam, \U \rangle &\leq 0 \quad &&\forall \Mu \in \bs{\varLambda}.
\end{alignedat}
\right.
\end{equation}
The existence, uniqueness and regularity of solutions of the variational problem (without the Lagrange multiplier) has been studied in \cite{Fujita1994} and \cite{Saito04}.
The inequality in \eqref{eq:weak1}
is equivalent to $-\langle \mu_t - \lambda_t, u_t \rangle \leq 0$ (when $\mu_n = \lambda_n$) which
follows from the definition of the
space $\bs{\varLambda}$
and the last equality in \eqref{eq:strongtresca}.

Defining the bilinear form
\begin{align*}
    \Bf(\W, r, \X; \V, q, \Mu) &= \Af(\W,\V) + (q, \Div \W) - (r, \Div \V) - \langle \X, \V \rangle - \langle \Mu, \W \rangle
\end{align*}
where
\[
\Af(\W, \V) = (\W, \V) + (2\mu\D(\W), \D(\V)),
\]
we can write the problem \eqref{eq:weak1}, by summing the three parts in \eqref{eq:weak1}, in the following compact form:
\begin{prob}[Continuous variational form]
  \label{prob:cont}
  Find $(\U, p, \Lam) \in \bs{V} \times Q \times \bs{\varLambda}$ such that
  \[
  \Bf(\U, p, \Lam; \V, q, \Mu - \Lam) \leq (\F, \V) \quad \forall (\V, q, \Mu) \in \bs{V} \times Q \times \bs{\varLambda}.
  \]
\end{prob}
 
The following norm will be used in our analysis:
\begin{equation}
    \label{eq:contnorm}
    \enorm{(\W,r,\X)}^2 = \| \W \|_{1,\Omega}^2 + \| r \|_{0,\Omega}^2 + \| \X \|_{-\frac12}^2,
\end{equation}
where $ \| \cdot \|_{1,\Omega}$ and $ \| \cdot \|_{0,\Omega}$ are the usual norms in the Hilbert spaces $H^1(\Omega)$ and $L^2(\Omega)$ and
\[
\| \X \|_{-\frac12} = \sup_{\bs{w} \in \bs{W}} \frac{\langle \X, \bs{w}\rangle}{ \| \bs{w} \|_{\frac12}} \quad \text{and} \quad \| \bs{w} \|_{\frac12} = \inf_{\substack{\bs{v} \in (H^1(\Omega))^d\\ \bs{v}|_{\partial \Omega} = \bs{w}}} \| \nabla \bs{v} \|_0.
\]
Note that there exist $C, c > 0$ such that
\begin{equation}
\label{eq:normequiv}
c \| \W \|_{1,\Omega}^2 \leq \Af(\W, \W) \leq C \| \W \|_{1,\Omega}^2.
\end{equation}
In the following, we write $a \lesssim b$ (or $a \gtrsim b$) if there exists a constant $C > 0$, which is  independent of the
finite element mesh, but possibly varying from step to step, and satisfies $a \leq C b$ (or $a \geq C b$).

The proof of the following result can be inferred, e.g., from~\cite{Ayadi2014}.
\begin{thm}[Continuous stability]
\label{thm:contstab}
For every $(\W, r, \X) \in \bs{V} \times Q \times \bs{\varLambda}$ there exists $(\V, q, \Mu) \in \bs{V} \times Q \times \bs{\varLambda}$
satisfying
\[
\Bf(\W, r, \X; \V, q, \Mu) \gtrsim \enorm{(\W, r, \X)}^2
\]
and
\[
\enorm{(\V, q, \Mu)} \lesssim \enorm{(\W, r, \X)}.
\]
\end{thm}

\section{Stabilised finite element method}

We consider finite element spaces based on a shape regular triangulation $\Th$ of $\Omega$
with the mesh parameter $h$.
We denote by $\Eh$ the internal facets and by $\Gh$  the boundary facets of $\Th$, respectively.
The finite element spaces are denoted by $\bs{V}_h \subset \bs{V}$, $Q_h \subset Q$, $\bs{M}_h \subset \bs{M}$,
and, in addition, we define the discrete counterpart of $\bs{\varLambda}$ as follows:
\[
\bs{\varLambda}_h = \{ \Mu \in \bs{M}_h : | \Mu_{t} | \leq \kappa \}.
\]
Our analysis is based on the conformity
assumption $\bs{\varLambda}_h \subset \bs{\varLambda}$
which means that we must be able to enforce
the condition $| \Lam_{h,t} | \leq \kappa$ strongly.  In practice, this means, e.g., that  $\kappa$ and $\Lam_h$ are constants elementwise.
We note that while the conformity is required
by our analysis, the resulting algorithm gives
reasonable results also for nonconstant $\kappa$
where the condition $| \Lam_{h,t} | \leq \kappa$ holds
only, e.g., at element midpoints or at the nodes of the mesh.

Let $\alpha_1, \alpha_2 > 0$ be stabilisation parameters.
The finite element method is written with the help of the stabilised bilinear form
\[
\Bf_h(\W, r, \X; \V, q, \Mu) = \Bf(\W, r, \X; \V, q, \Mu) - \alpha_1 \Sf^1(\W, r; \V, q) - \alpha_2 \Sf^2(\W, r, \X; \V, q, \Mu)
\]
where
\[
\Sf^1(\W, r; \V, q) = \sum_{T \in \Th} h_T^2 \int_{T} (\A \W + \nabla r) \cdot (\A \V - \nabla q) \,\mathrm{d}x
\]
and
\[
\Sf^2(\W, r, \X; \V, q, \Mu) = \sum_{E \in \Gh} h_E \int_{E} (\X - \sig(\W, r)\N) \cdot( \Mu - \sig(\V, q)\N) \,\mathrm{d}s.
\]
The stabilised linear form is given by
\[
\Lf_h(\V, q) = (\F, \V) - \alpha_1 \Fh(\V, q)
\]
where
\[
\Fh(\V, q) = \sum_{T \in \Th} h_T \int_{T} \F \cdot ( \A \V - \nabla q) \,\mathrm{d}x.
\]
This type of nonsymmetric residual stabilisation for the Stokes operator can be found, e.g., in \cite{barth2004taxonomy}.

The stabilised finite element method corresponds to solving the following variational problem.
\begin{prob}[Discrete variational form]
  \label{prob:disc}
    Find $(\U_h, p_h, \Lam_h) \in \bs{V}_h \times Q_h \times \bs{\varLambda}_h$ such that
\begin{equation*}
  \Bf_h(\U_h, p_h, \Lam_h; \V_h, q_h, \Mu_h - \Lam_h) \leq \Lf_h(\V_h, q_h) \quad \forall (\V_h, q_h, \Mu_h) \in \bs{V}_h \times Q_h \times \bs{\varLambda}_h.
\end{equation*}
\end{prob}

In our analysis, we will use the following inverse and trace estimates, easily proven by a scaling argument~\cite{stenberg2015error}.
\begin{lem}[Inverse estimates]
\label{lem:inverse}
For any $(\W_h, r_h) \in \bs{V}_h \times Q_h$,
there exist constants $C_{I,1}, C_{I,2} > 0$ such that
\[
C_{I,1} \sum_{T \in \Th} h_T^2 \| \A \W_h \|_{0,T}^2 \leq a(\W_h, \W_h)
\]
and
\[
C_{I,2} \sum_{E \in \Gh} h_E \| \sig(\W_h, r_h) \N \|_{0,E}^2 \leq a(\W_h, \W_h).
\]
\end{lem}
\begin{lem}[Discrete trace estimate]
\label{lem:disctrace}
For any $\W_h \in \bs{V}_h$, there exists $C_T>0$ such that
\[
C_T \sum_{E \in \Gh} h_E\|2\mu \D(\W_h)\N\|_{0,E}^2\leq \|\W_h\|_{1,\Omega}^2 .
\]
\end{lem}
The following discrete counterpart of the norm defined in \eqref{eq:contnorm} is instrumental in the stability analysis of the discrete problem.
\begin{equation}
    \label{eq:discnorm}
    \enorm{(\W_h,r_h,\X_h)}^2_h = \enorm{(\W_h,r_h,\X_h)}^2 + \sum_{T \in \Th} h_T^2 \| \nabla r_h \|_{0,T}^2 + \sum_{E \in \Gh} h_E \| \X_h \|^2_{0,E}.
\end{equation}
Note, in particular, that
\begin{equation*}
    \enorm{(\W_h,r_h,\X_h)}_h \geq \enorm{(\W_h,r_h,\X_h)}.
\end{equation*}

Existence and uniqueness of solutions to the discrete variational problem follows from the discrete stability estimate proven below.

\begin{thm}[Discrete stability]
\label{thm:discstab}
Let $\frac{\alpha_1}{C_{I,1}} + \frac{\alpha_2}{C_{I,2}} < 1$.
For every $(\W_h, r_h, \X_h) \in \bs{V}_h \times Q_h \times \bs{\varLambda}_h$
there exists $(\V_h, q_h) \in \bs{V}_h \times Q_h$
satisfying
\begin{equation}
\label{eq:discstab}
\Bf_h(\W_h, r_h, \X_h; \V_h, q_h, -\X_h) \gtrsim \enorm{(\W_h, r_h, \X_h)}_h^2
\end{equation}
and
\[
\enorm{(\V_h, q_h, \Mu_h)}_h \lesssim \enorm{(\W_h, r_h, \Mu_h)}_h \quad \forall \Mu_h \in \bs{\varLambda}_h.
\]
\end{thm}

\begin{proof}
(Step 1.) Choosing $(\V_h, q_h, \Mu_h) = (\W_h, r_h, -\X_h)$ gives
\begin{align*}
&\Bf_h(\W_h, r_h, \X_h; \W_h, r_h, -\X_h)\\
&=a(\W_h, \W_h) - \alpha_1 \sum_{T \in \Th} h_T^2 \| \A \W_h \|_{0,T}^2  - \alpha_2 \sum_{E \in \Gh} h_E \| \sig(\W_h, r_h) \N \|_{0,E}^2\\
&\quad + \alpha_1 \sum_{T \in \Th} h_T^2 \| \nabla r_h \|_{0,T}^2 + \alpha_2 \sum_{E \in \Gh} h_E \| \X_h \|^2_{0,E}.
\end{align*}
Using the inverse estimates of Lemma~\ref{lem:inverse} and the bound \eqref{eq:normequiv} we get
\begin{align*}
&\Bf_h(\W_h, r_h, \X_h; \W_h, r_h, -\X_h)\\
&\geq \left(1 - \frac{\alpha_1}{C_{I,1}} - \frac{\alpha_2}{C_{I,2}}\right)c \|\W_h\|_{1,\Omega}^2 + \alpha_1 \sum_{T \in \Th} h_T^2 \| \nabla r_h \|_{0,T}^2 + \alpha_2 \sum_{E \in \Gh} h_E \| \X_h \|^2_{0,E}.
\end{align*}

(Step 2.) As a consequence of Theorem~\ref{thm:contstab}, for any $(r_h, \X_h) \in Q_h \times \bs{\varLambda}_h$
there exists $\V \in \bs{V}$ such that
\begin{equation}
\label{eq:continfsup}
(r_h, \Div \V) + \langle \X_h, \V \rangle \geq C_1 (\|r_h\|_{0,\Omega}^2 + \| \X_h \|_{-\frac{1}{2}}^2)
\end{equation}
and
\begin{equation}
    \| \V \|_{1,\Omega}^2 \leq C_2( \|r_h\|_{0,\Omega}^2 + \| \X_h \|_{-\frac{1}{2}}^2)
\label{eq:vbound}
\end{equation}
where $C_1, C_2 > 0$.
Let $\widetilde{\V}\in \bs{V}_h$ be the Clem\'{e}nt interpolant of $\V$
with the properties
\begin{equation}
\label{eq:clement1}
\left(\sum_{T \in \Th} h_T^{-2} \| \V - \widetilde{\V} \|_{0,T}^2\right)^{1/2} + \left(\sum_{E \in \Gh} h_E^{-1} \| \V - \widetilde{\V} \|_{0,E}^2\right)^{1/2} \leq C_{i,1} \| \V \|_{1,\Omega}
\end{equation}
and
\begin{equation}
\label{eq:clement2}
\| \widetilde{\V} \|_{1,\Omega} \leq C_{i,2} \| \V \|_{1,\Omega}.
\end{equation}
Choosing $(\V_h, q_h, \Mu_h) = (-\widetilde{\V}, 0, \bs{0})$ gives
\begin{equation}
\label{eq:step2first}
\begin{aligned}
&\Bf_h(\W_h, r_h, \X_h; -\widetilde{\V}, 0, \bs{0})\\
&= -a(\W_h, \widetilde{\V}) + (r_h, \Div \widetilde{\V}) + \langle \X_h, \widetilde{\V} \rangle
+ \alpha_1 \Sf^1(\W_h, r_h; \widetilde{\V}, 0)\\
&\quad+ \alpha_2 \Sf^2(\W_h, r_h, \X_h; \widetilde{\V}, 0, \bs{0}).
\end{aligned}
\end{equation}
The first term in \eqref{eq:step2first} can be bounded using the continuity of  $a$, Young's inequality with a constant $\delta_1 > 0$, and the interpolation property \eqref{eq:clement2}. This leads to the bound
\begin{equation}
- \frac{C_\text{cont.} \delta_1}{2} \| \W_h \|_{1,\Omega}^2 - \frac{C_\text{cont.} C_{i,2}}{2\delta_1} \| \V \|_{1,\Omega}^2.
\end{equation}
The second and the third terms are bounded using integration by parts, Cauchy--Schwarz inequality, the bound \eqref{eq:continfsup}
and the properties of the Clem\'ent interpolant as follows:
\begin{equation}
\begin{aligned}
&(r_h, \Div \widetilde{\V}) + \langle \X_h, \widetilde{\V} \rangle \\
&=-(r_h, \Div (\V-\widetilde{\V})) - \langle \X_h,\V- \widetilde{\V} \rangle + (r_h, \Div \V) + \langle \X_h, \V \rangle \\
&\geq - \sum_{T \in \Th} h_T \| \nabla r_h \|_{0,T} h_T^{-1} \| \V - \widetilde{\V} \|_{0,T} - \sum_{E \in \Gh} h_E^{1/2} \| \X_h \|_{0,E} h_E^{-1/2} \| \V - \widetilde{\V} \|_{0,E} \\
&\qquad + C_1(\|r_h\|_{0,\Omega}^2 + \| \X_h \|_{-\frac12}^2) \\
&\geq - \left(\left(\sum_{T \in \Th} h_T^2 \| \nabla r_h \|_{0,T}^2\right)^{1/2}+ \left(\sum_{E \in \Gh} h_E \| \X_h \|_{0,E}^2\right)^{1/2}\right) C_{i,1} \| \V \|_{1,\Omega}  \\
&\qquad + C_1(\|r_h\|_{0,\Omega}^2 + \| \X_h \|_{-\frac12}^2).
\end{aligned}
\end{equation}
After applying Young's inequality with  constants $\delta_2, \delta_3 > 0$, we finally obtain
\begin{equation}
\begin{aligned}
&(r_h, \Div \widetilde{\V}) + \langle \X_h, \widetilde{\V} \rangle \\
&\geq - \frac{C_{i,1} \delta_2}{2} \sum_{T \in \Th} h_T^2 \| \nabla r_h \|_{0,T}^2 - \frac{C_{i,1} \delta_3}{2} \sum_{E \in \Gh} h_E \| \X_h \|_{0,E}^2  \\
&\qquad - \left(\frac{C_{i,1}}{2 \delta_2}
 + \frac{C_{i,1}}{2 \delta_3} \right) \| \V \|_{1,\Omega}^2 + C_1(\|r_h\|_{0,\Omega}^2 + \| \X_h \|_{-\frac12}^2).
\end{aligned}
\end{equation}
Next, we bound the two stabilisation
terms in \eqref{eq:step2first}. The first can be bounded from below as follows:
\[
\begin{aligned}
&\alpha_1 \Sf^1(\W_h, r_h; \widetilde{\V}, 0) = \alpha_1 \sum_{T \in \Th} h_T^2(A\W_h+\nabla r_h,A\widetilde{\V})_T\\
 &\qquad \geq -\alpha_1 \left( \sum_{T \in \Th} h_T \|A\W_h\|_{0,T}+\sum_{T \in \Th} h_T \|\nabla r_h\|_{0,T}\right)\, \sum_{T \in \Th} h_T\|A\widetilde{\V}\|_{0,T}\\
 &\qquad \geq -\alpha_1 \left(\left( \sum_{T \in \Th} h_T^2 \|A\W_h\|_{0,T}^2\right)^{1/2}+\left(\sum_{T \in \Th} h_T^2\|\nabla r_h\|_{0,T}^2\right)^{1/2}\right)\\
 &\qquad \qquad \qquad \cdot\left(\sum_{T \in \Th} h_T^2\|A\widetilde{\V}\|_{0,T}^2\right)^{1/2}.
\end{aligned}
\]
Given that
\begin{equation}
\sum_{T \in \Th} h_T^2 \|A\W_h\|_{0,T}^2 \leq C_{I,1}^{-1} a(\W_h,\W_h)\leq CC_{I,1}^{-1} \|\W_h\|_{1,\Omega}^2,
\end{equation}
and similarly for $\widetilde{\V}$, we conclude, using Young's inequality, \eqref{eq:vbound} and \eqref{eq:clement2}, that
\begin{equation}
\begin{aligned}
&\alpha_1 \Sf^1(\W_h, r_h; \widetilde{\V}, 0)\\
&\qquad \geq -\alpha_1 CC_{I,1}^{-1}\left(
\frac{\delta_4}{2} CC_{I,1}^{-1} \|\W_h\|_{1,\Omega}^2+\frac{\delta_5}{2} \sum_{T \in \Th} h_T^2\|\nabla r_h\|_{0,T}^2 \right.\\
&\qquad \hspace{2.5cm} \left. +C_{i,2}^2C_2\left(\frac{1}{2\delta_4}+\frac{1}{2\delta_5}\right) \left(  \|r_h\|_{0,\Omega}^2 + \| \X_h \|_{-\frac{1}{2}}^2\right)
\right).
\end{aligned}
\end{equation}
The second stabilisation term in \eqref{eq:step2first}
 is bounded using Lemmas \ref{lem:inverse}  and \ref{lem:disctrace}, Young's inequality as well as  bounds \eqref{eq:vbound} and \eqref{eq:clement2}.
\begin{equation*}
\begin{aligned}
& \alpha_2 \,\Sf^2(\W_h, r_h, \X_h; \widetilde{\V}, 0, \bs{0}) = \alpha_2\sum_{E \in \Gh} h_E\left(\X_h-\sig(\W_h,r_h)\N,-\sig(\widetilde{\V},0)\N\right)_E\\
&\qquad =- \alpha_2\sum_{E \in \Gh} h_E\left(\X_h,2\mu \D(\widetilde{\V})\N\right)_E+\alpha_2 \sum_{E \in \Gh} h_E\left(\sig(\W_h,r_h)\N,2\mu \D(\widetilde{\V})\N\right)_E\\
&\qquad \geq -\alpha_2    \frac{\delta_6}{2} \sum_{E \in \Gh} h_E\|\X_h\|_{0,E}^2 - \alpha_2 C_{I,2}^{-1}C   \frac{\delta_7}{2}\|\W_h\|_{1,\Omega}^2 \\ 
&\qquad \qquad \qquad + C_{i,2}^2C_2C_T^{-1} \left(\frac{1}{2\delta_6}+\frac{1}{2\delta_7}\right) \left( \|r_h\|_{0,\Omega}^2 + \| \X_h \|_{-\frac{1}{2}}^2\right)
\end{aligned}
\end{equation*}
(Step 3.) Finally, we combine steps 1 and 2
by showing that if we choose $(\V_h, q_h, \Mu_h) = (\W_h - \varepsilon \widetilde{\V}, r_h, -\X_h)$, we can  guarantee  that $\varepsilon > 0$
and the other  constants $\delta_i$, $i=1,\dots,7$, remaining from application of  Young's
inequalities can be chosen in such a way
that the coefficients of the terms
comprising the norm $\enorm{.}_h$ remain positive.
\end{proof}

\section{Error analysis}

Let   $\F_h$ be the $L^2$-projection of $\F$ onto $\bs{V}_h$ and, for $T\in  \Th $, define
\[\mathrm{osc}_T\F=h_T \| \F - \F_h \|_{0,T},
 \qquad \mathrm{osc}\,\F = \left(\sum_{T \in \Th} (\mathrm{osc}_{T}\F)^2 \right)^{1/2}.
\]
Below we denote by $T(E) \in \Th$ the element
which has $E \in \Gh$ as one of its facets.
The proofs of the following lemmas can be found, e.g.,
in \cite{gustafsson2022stabilized} and \cite{verfurth1989posteriori,VerfurthBook}.

\begin{lem}[Lower bound for the boundary residual]
\label{lem:resbound1}
For any $(\W_h, r_h, \X_h) \in \bs{V}_h \times Q_h \times \bs{\varLambda}_h$ it holds that
\begin{align*}
    &\left(\sum_{E \in \Gh} h_E \| \X_h - \sig(\W_h, r_h)\N \|_{0,E}^2\right)^{1/2} \\
    &\qquad \lesssim \enorm{(\U - \W_h, p - r_h, \Lam - \X_h)} + \left(\sum_{E \in \Gh} (\mathrm{osc}_{T(E)}\F)^2 \right)^{1/2}.
\end{align*}
where $(\U, p, \Lam) \in \bs{V} \times Q \times \bs{\varLambda}$ is the solution to \eqref{eq:weak1}.
\end{lem}

\begin{lem}[Lower bound for the interior residual]
\label{lem:resbound2}
For any $(\W_h, r_h) \in \bs{V}_h \times Q_h$ it holds that
\begin{align*}
    \left(\sum_{T \in \Th} h_T^2 \| \A \W_h + \nabla r_h + \F \|_{0,T}^2\right)^{1/2} \lesssim \| \U - \W_h \|_1 + \|p - r_h\|_0 + \mathrm{osc}\,\F
\end{align*}
where $(\U, p) \in \bs{V} \times Q$ is the solution to \eqref{eq:weak1}.
\end{lem}

We can now show the quasi-optimality of the method, i.e.~a best approximation
result with an additional term due to the inequality constraint.
\begin{thm}[Quasi-optimality]
\label{thm:quasi}
For any $(\W_h, r_h, \Mu_h) \in \bs{V}_h \times Q_h \times \bs{\varLambda}_h$,
it holds
\begin{equation}
\label{eq:thmquasi}
\begin{aligned}
&\enorm{(\U - \U_h, p - p_h, \Lam - \Lam_h)}\\
&\quad \lesssim \enorm{(\U - \W_h, p - r_h, \Lam - \Mu_h)} + \sqrt{\langle \Mu_h - \Lam, \U \rangle} + \mathrm{osc}\,\F,
\end{aligned}
\end{equation}
where $(\U_h, p_h, \Lam_h)$ denotes the solution to Problem \ref{prob:disc}.
\end{thm}
\begin{proof}
Let $(\W_h, r_h, \Mu_h) \in \bs{V}_h \times Q_h \times \bs{\varLambda}_h$ be arbitrary. Then by the discrete stability estimate \eqref{eq:discstab} there exists $(\V_h, q_h) \in \bs{V}_h \times Q_h$,
with the property
$\enorm{(\V_h, q_h, \Lam_h - \Mu_h)}_h = 1$, such that
\begin{equation}
\label{quasiproof1}
\begin{aligned}
    &\enorm{(\U_h - \W_h, p_h - r_h, \Lam_h - \Mu_h)}_h\\
    &\lesssim \Bf_h(\U_h - \W_h, p_h - r_h, \Lam_h - \Mu_h; \V_h, q_h, \Mu_h - \Lam_h) \\
    &\lesssim \Bf_h(\U_h, p_h, \Lam_h; \V_h, q_h, \Mu_h - \Lam_h) - \Bf_h(\W_h, r_h, \Mu_h; \V_h, q_h, \Mu_h - \Lam_h) \\
    &\lesssim \Lf_h(\V_h, q_h) - \Bf(\U, p, \Lam; \V_h, q_h, \Mu_h - \Lam_h) \\
    &\quad + \Bf(\U - \W_h, p - r_h, \Lam - \Mu_h; \V_h, q_h, \Mu_h - \Lam_h)\\
    &\quad  + \alpha_1 \Sf^1(\W_h, r_h; \V_h, q_h) + \alpha_2 \Sf^2(\W_h, r_h, \Mu_h; \V_h, q_h, \Mu_h - \Lam_h) ,
\end{aligned}
\end{equation}
where in the last step  we have used the discrete variational form and written out the discrete bilinear form. 
The first two terms on the right-hand side of \eqref{quasiproof1} can be written as
\begin{align*}
&\Lf(\V_h, q_h) - \Bf(\U, p, \Lam; \V_h, q_h, \Mu_h - \Lam_h) 
- \alpha_1 \Fh(\V_h, q_h) \\
&= ( \U, \Mu_h - \Lam_h)_{\partial \Omega}- \alpha_1 \Fh(\V_h, q_h) \\
&\leq \langle \Mu_h - \Lam, \U \rangle - \alpha_1 \Fh(\V_h, q_h)
\end{align*}
where the inequality follows from the inequality in \eqref{eq:weak1}.
The third term is bounded
using the continuity of the bilinear form $\Bf$ and
the final two stabilisation terms are bounded using Cauchy--Schwarz inequality and Lemmas~\ref{lem:inverse}, \ref{lem:resbound1} and \ref{lem:resbound2}.
The proof is completed using the trivial bound
\[
\enorm{(\U_h - \W_h, p_h - r_h, \Lam_h - \Mu_h)} \leq \enorm{(\U_h - \W_h, p_h - r_h, \Lam_h - \Mu_h)}_h
\]
and the triangle inequality.
\end{proof}

\newcommand*{\PiV}{\Pi_{\bs{V}_h}}
\newcommand*{\PiQ}{\Pi_{Q_h}}
\newcommand*{\PiM}{\Pi_{\bs{M}_h}}

Using a continuous pressure, we can consider the finite
element spaces 
\begin{align}
\bs{V}_h &= \{ \V \in \bs{V} : \V|_{T} \in P_k(T)^d~\forall T \in \Th \},\\
Q_h &= \{ q \in Q \cap C(\Omega) : q|_T \in P_l(T)~\forall T \in \Th \},\\
\bs{M}_h &= \{ \Mu \in \bs{M} : \Mu|_E \in P_m(E)^d~\forall E \in \Gh \},
\end{align}
where $k,l\geq1$ and $m\geq0$
are the polynomial
orders.
Alternatively, we may consider
a discontinuous pressure
which, however, requires
a quadratic velocity, $k \geq 2$.
It is also possible
to use a continuous Lagrange multiplier
together with any valid velocity--pressure combination.

\begin{rem}
The analysis up to this
point is valid for $d \in \{ 2, 3 \}$
and for any $k, l \geq 1$ and $m \geq 0$.
However, proving an optimal a priori
error estimate
based on the
quasi-optimality result shown in
Theorem~\ref{thm:quasi}
requires further assumptions.
For instance, in the two-dimensional case, assuming that $\Lam \in H^{1/2}(\partial \Omega)^2$, then
$|\Lam_t| \leq \kappa$ holds
pointwise and it should be clear that also $|(\PiM \Lam)_t| \leq \kappa$
where $\PiM$ is the $L^2$-projection onto $\bs{M}_h$.
This implies that $\PiM \Lam \in \bs{\varLambda}_h$
and, consequently, allows us to write the bound
\[
\| \Lam - \Mu_h \|_{-\frac12} = \| \Lam - \PiM \Lam \|_{-\frac12} \leq \| \Lam - \PiM \Lam \|_{0,\Omega} \lesssim h.
\]
Assuming, moreover, that $\U \in H^2(\Omega)^2$, one thus obtains an optimal a priori estimate for the lowest order elements given that
\begin{align*}
&\sqrt{(\Mu_h - \Lam, \U )_{\partial \Omega}} \\
&=\sqrt{(\PiM \Lam - \Lam, \U )_{\partial \Omega}} \\
&= \sqrt{(\PiM \Lam - \Lam, \U -\PiV \U)_{\partial \Omega}} \\
&\leq \sqrt{\|\Lam - \PiM \Lam\|_{0,\partial \Omega} \|\U -\PiV \U\|_{0,\partial \Omega}} \lesssim h,
\end{align*}
where $\PiV$ denotes the Lagrange interpolant onto $\bs{V}_h$. 
\end{rem}

\section{Solution algorithm}

We next derive our solution algorithm,
also known as Uzawa iteration,
following the steps given in
He--Glowinski~\cite{he2000steady}.
The derivation is given in detail because
the algorithm includes additional
terms due to the stabilisation.

The discrete variational problem can be split into
\begin{equation}
\label{eq:eqpart}
\Bf_h(\U_h, p_h, \Lam_h; \V_h, q_h, \Z) = \Lf_h(\V_h, q_h)
\end{equation}
and
\begin{equation}
\label{eq:ineqpart}
- \int_{\partial \Omega} \U_h \cdot( \Mu_h - \Lam_h )\,\mathrm{d}s -  \alpha_2 \sum_{E \in \Gh} h_E \int_{E} (\Lam_h - \sig(\U_h, p_h)\N) \cdot (\Mu_h - \Lam_h)  \,\mathrm{d}s  \leq 0.
\end{equation}
Combining the two terms we equivalently have
\begin{equation*}
-\int_{\partial \Omega} \Pi_{\bs{M}_h}(\U_h +  \alpha_2 \hf(\Lam_h - \sig(\U_h, p_h)\N)) \cdot (\Mu_h - \Lam_h)  \,\mathrm{d}s  \leq 0,
\end{equation*}
where $\hf$ is the boundary mesh size function and $\PiM$ is the
$L^2$-projection onto $\bs{M}_h$.
Now multiplying by an arbitrary $\rho>0$, and adding and subtracting $\Lam_h$
leads to
\begin{equation*}
\int_{\partial \Omega} (\Lam_h-\rho\,\Pi_{\bs{M}_h}(\U_h +  \alpha_2 \hf(\Lam_h - \sig(\U_h, p_h)\N)) - \Lam_h) \cdot (\Mu_h - \Lam_h)  \,\mathrm{d}s  \leq 0.
\end{equation*}
The above form implies that $\Lam_h$ is equal to the orthogonal projection of
\[
\Lam_h-\rho\,\Pi_{\bs{M}_h}(\U_h +  \alpha_2 \hf(\Lam_h - \sig(\U_h, p_h)\N))
\]
onto the constrained space $\bs{\varLambda}_h$.
The orthogonal projection can be written explicitly as
\[
\bs{P}(\X) = (\X \cdot \N) \N + \frac{\kappa (\X - (\X \cdot \N) \N)}{\max(\kappa, |\X - (\X \cdot \N) \N|)}
\]
which can be interpreted as enforcing the maximum length of the tangential component to $\kappa$.
As a conclusion, the inequality constraint \eqref{eq:ineqpart} can be reformulated as the equality constraint
\begin{equation}
\label{eq:returnmap}
\Lam_h = \bs{P}(\Lam_h-\rho\,\Pi_{\bs{M}_h}(\U_h +  \alpha_2 \hf(\Lam_h - \sig(\U_h, p_h)\N)))~~\text{a.e.~on $\partial \Omega$}.
\end{equation}
 
\begin{algo}[Uzawa iteration]
  \label{algo:uzawa}
  Let $(\U_h^0, p_h^0, \Lam_h^0)$
  be an initial guess, $TOL > 0$ be a stopping tolerance and set $i \leftarrow 1$.
  \begin{enumerate}
  \item[1.] Calculate $\Lam_h^i = \bs{P}(\Lam_h^{i-1} - \rho \, \Pi_{\bs{M}_h}(\U_h^{i-1} + \alpha_2 \hf(\Lam_h^{i-1} - \sig(\U_h^{i-1}, p_h^{i-1})\N))$.
  \item[2.] Solve for $(\U_h^i, p_h^i)$
    in $\Bf_h(\U^i_h, p_h^i, \Z; \V_h, q_h, \Z) = \Lf_h(\V_h, q_h) + \langle \Lam_h^i, \V_h \rangle$.
  \item[3.] Stop if $\|\Lam_h^i - \Lam_h^{i-1}\|_0 / \| \Lam_h^i \|_0 < TOL$.  Otherwise set $i \leftarrow i + 1$ and go to Step~1.
  \end{enumerate}
\end{algo}

\begin{rem}
  An assumption is made in
  Step 1 of Algorithm~\ref{algo:uzawa}
  that the orthogonal projection $\bs{P}$ can be performed
  directly on the discrete function.
  This is true, e.g., if the Lagrange multiplier is
  approximated by a piecewise constant function
  or a discontinuous, piecewise linear function. 
\end{rem}

\section{Numerical experiment}

The numerical results are calculated with the
help of the software package scikit-fem~\cite{skfem2020}.
The source code
is available in \cite{gustafsson_tom_2023_8296578}.

\subsection{Convergence study}

In the first example, we consider the lowest order
method with $k=l=1$, $m=0$,
and solve the problem within
the domain
$\Omega = (-1,1)^2$ using
the material parameter values
$\kappa = 0.3$, $\mu=1$, and the loading function
\[
\F = (-y, x).
\]
We have set the numerical parameters to
$\rho=0.4$, $\alpha_1=\alpha_2=10^{-2}$,
$TOL=10^{-5}$
and solve the same problem
using a sequence of uniformly
refined meshes; cf.~Figure~\ref{fig:uniformmeshes}.
The Uzawa parameter $\rho=0.4$
is close to its upper bound (see, e.g., \cite{he2000steady} for more information) which
we have found by trial-and-error
in order to
reduce the number of iterations
required.
The stabilisation parameters
$\alpha_1$ and $\alpha_2$ have
been chosen, for simplicity,
to be equal while $\alpha_2$
is close to its upper bound,
as given by Theorem~\ref{thm:discstab}.

The components $u_h$ and $p_h$ of the discrete solution,
calculated using the finest mesh
in the sequence,
are visualized in Figure~\ref{fig:discsolsup}
while the discrete Lagrange multipliers
for some of the meshes
are given in Figure~\ref{fig:discsolslag}. As seen in the Figures, the fluid, which is flowing counterclockwise, is slipping along the middle part of all four sides of the boundary. 

In the absence of an analytical solution,
we have calculated the relative
errors in the discrete solutions
between two subsequent meshes
in Figures~\ref{fig:uerr}, \ref{fig:perr}
and \ref{fig:lamerr}.
Note that the relative error
can be shown to converge
at similar rates as the absolute error
by using the triangle inequality,
e.g.,
\[
\| \U_{2h} - \U_h \|_{1,\Omega}
= \| \U_{2h} - \U + \U - \U_h \|_{1,\Omega}
\leq \| \U_{2h} - \U \|_{1,\Omega} + \|\U - \U_h \|_{1,\Omega} \leq Ch.
\]
The numerical results suggest that the
total error converges linearly
as the observed rate is linear for the
velocity and the pressure, and
superlinear for
the Lagrange multiplier.

\subsection{Curved boundary}

In the second example, we demonstrate how the lowest order method performs with
curved boundaries.
The domain is now chosen as the half unit circle
\[
\Omega = \{ (x, y) : x^2 + (y - 0.5)^2 < 1, y < 0.5 \},
\]
and the parameters $\kappa = \rho = 0.1$.
The other parameters are as in the previous example.
The resulting velocity and pressure
fields are depicted in Figure~\ref{fig:discsolsupcircle}.
The tangential Lagrange multipliers are depicted in Figure~\ref{fig:lagmultscircle}.

\newcommand*{\vcenteredhbox}[1]{\begin{tabular}{@{}c@{}}#1\end{tabular}}

\begin{figure}
\centering
\includegraphics[width=0.32\textwidth]{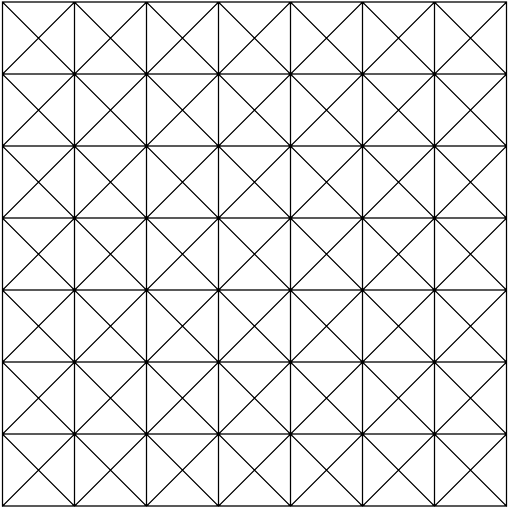}
\includegraphics[width=0.32\textwidth]{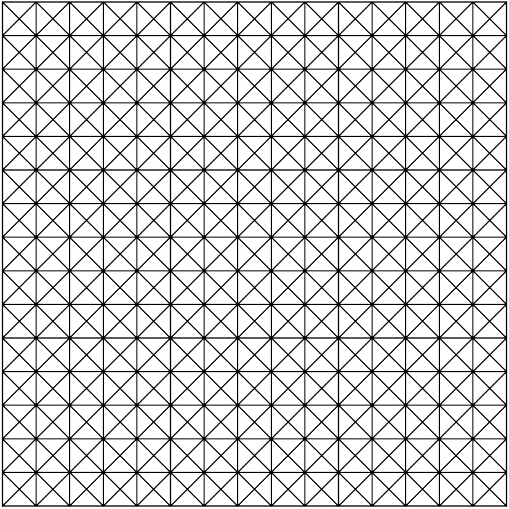}
\includegraphics[width=0.32\textwidth]{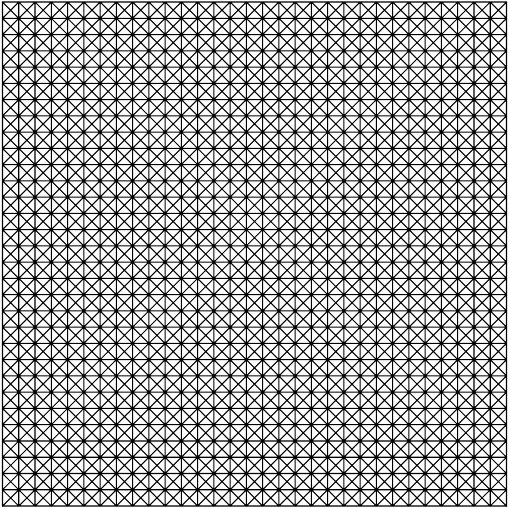}
\caption{Some meshes from the uniform sequence.}
\label{fig:uniformmeshes}
\end{figure}

\begin{figure}
\centering
\vcenteredhbox{\includegraphics[width=0.80\textwidth]{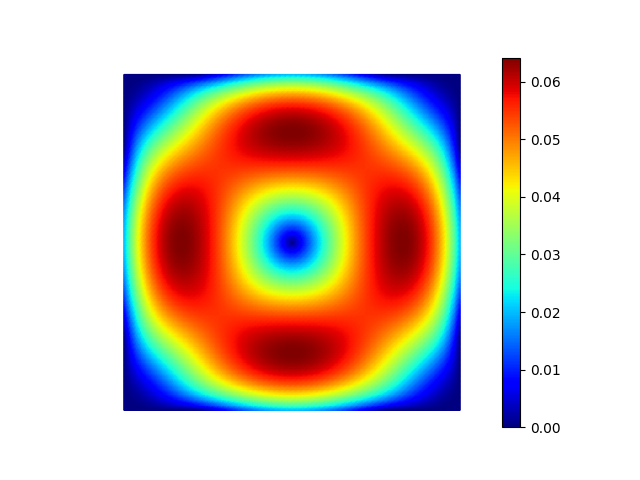}}
\vcenteredhbox{\includegraphics[width=0.80\textwidth]{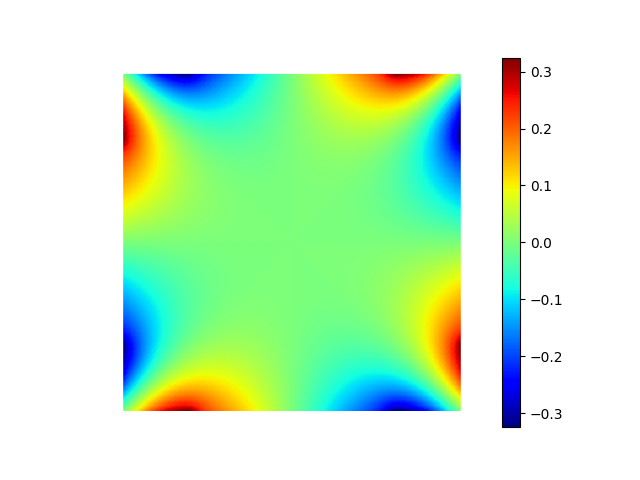}}
\caption{The velocity magnitude (top)
and the pressure (bottom)
computed using
the finest mesh in the uniform sequence.}
\label{fig:discsolsup}
\end{figure}

\begin{figure}
\centering
\includegraphics[width=0.4\textwidth]{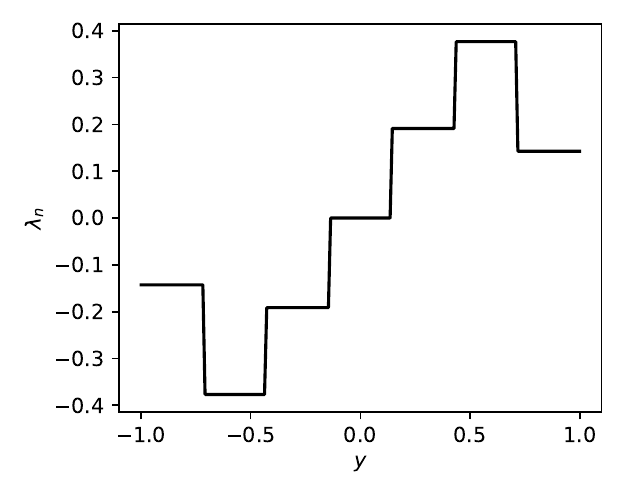}
\includegraphics[width=0.4\textwidth]{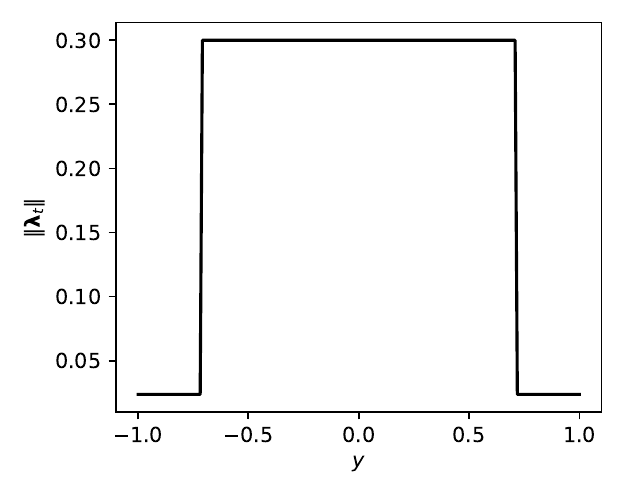}
\includegraphics[width=0.4\textwidth]{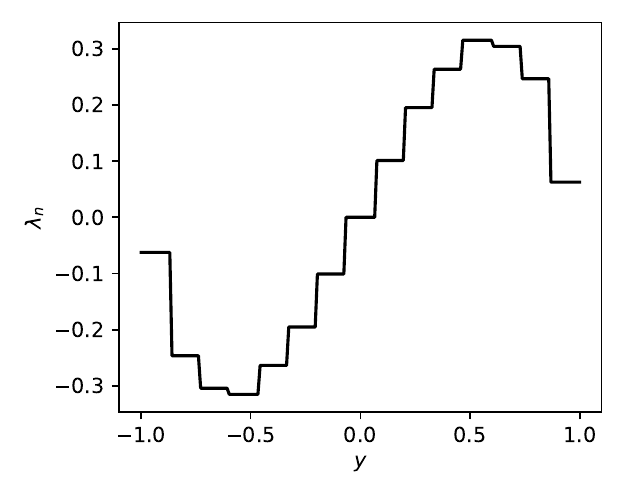}
\includegraphics[width=0.4\textwidth]{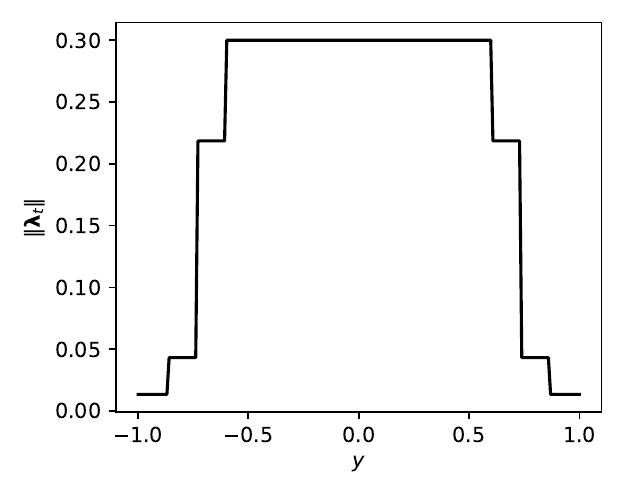}
\includegraphics[width=0.4\textwidth]{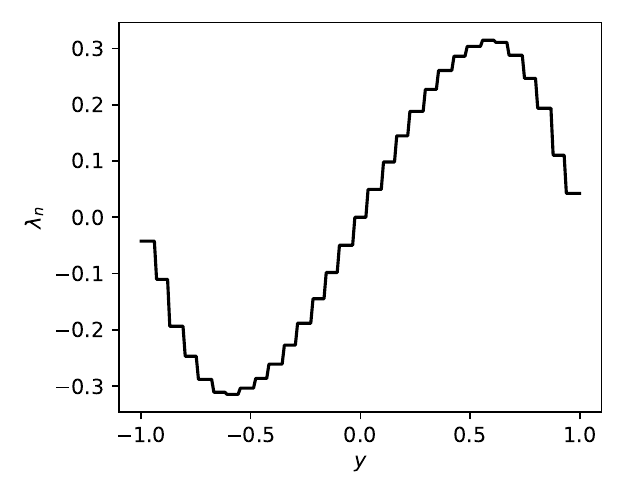}
\includegraphics[width=0.4\textwidth]{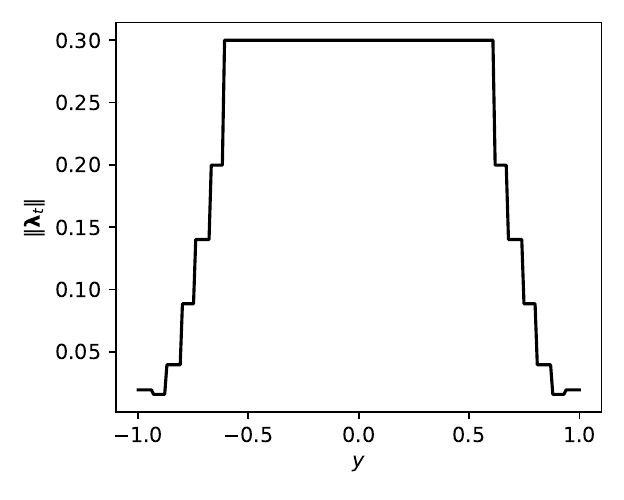}
\includegraphics[width=0.4\textwidth]{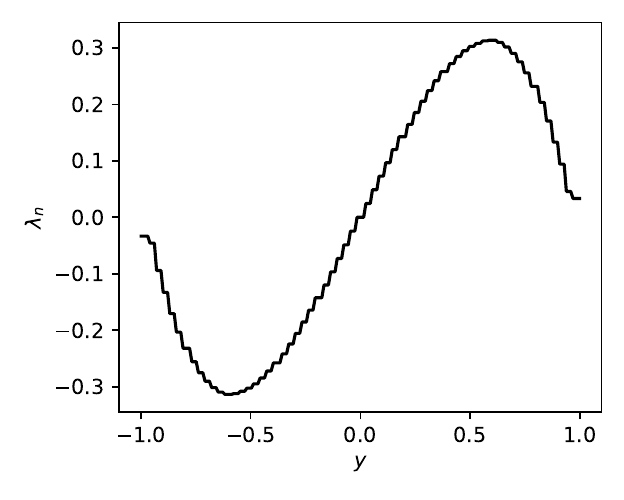}
\includegraphics[width=0.4\textwidth]{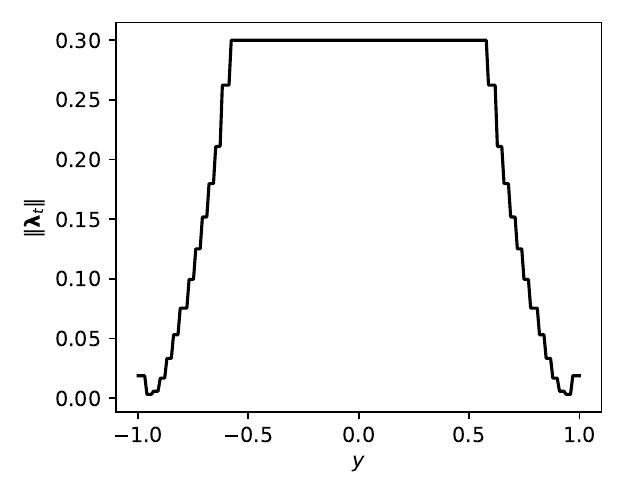}
\caption{The components of the discrete Lagrange multiplier solution at the boundary $x=1$ plotted for four different meshes from coarsest (top) to finest (bottom).}
\label{fig:discsolslag}
\end{figure}

\begin{figure}
\centering
\includegraphics[width=0.8\textwidth]{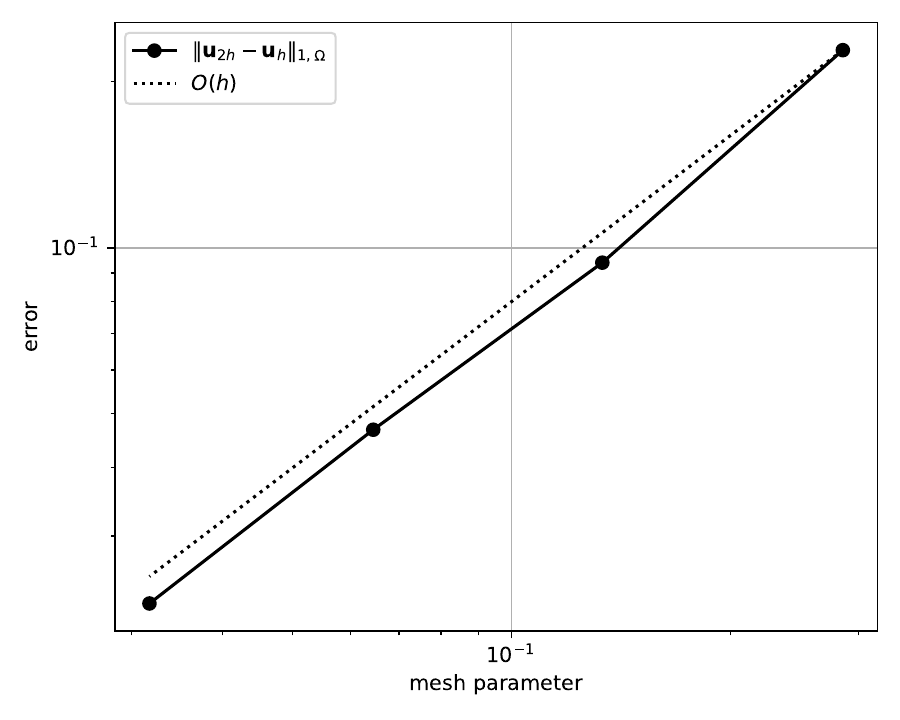}
\caption{The convergence of the error in the velocity.}
\label{fig:uerr}
\end{figure}
\begin{figure}
\centering
\includegraphics[width=0.8\textwidth]{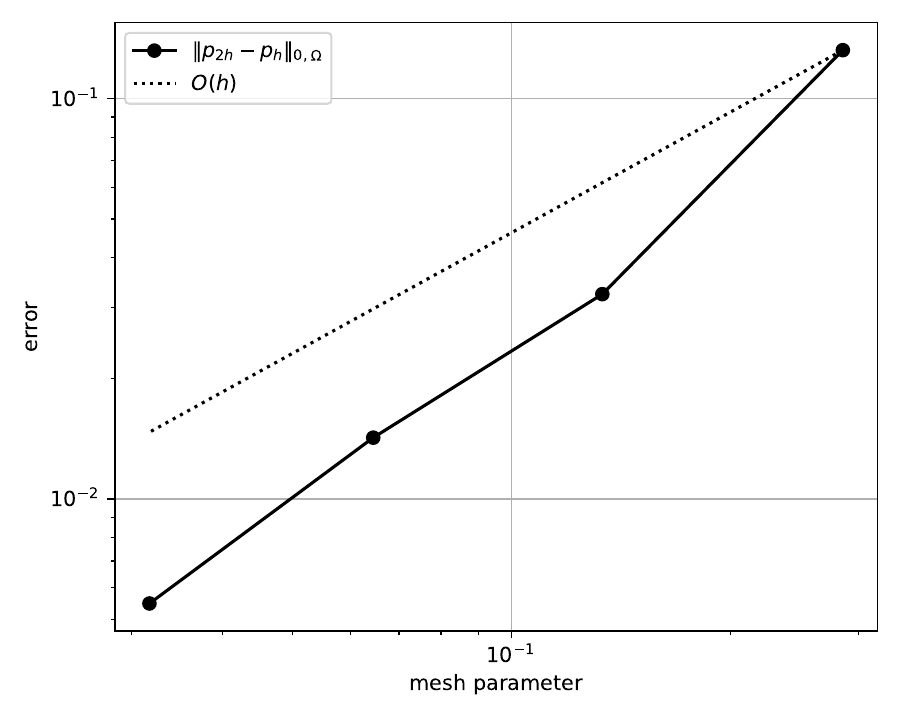}
\caption{The convergence of the error in the pressure.}
\label{fig:perr}
\end{figure}
\begin{figure}
\centering
\includegraphics[width=0.8\textwidth]{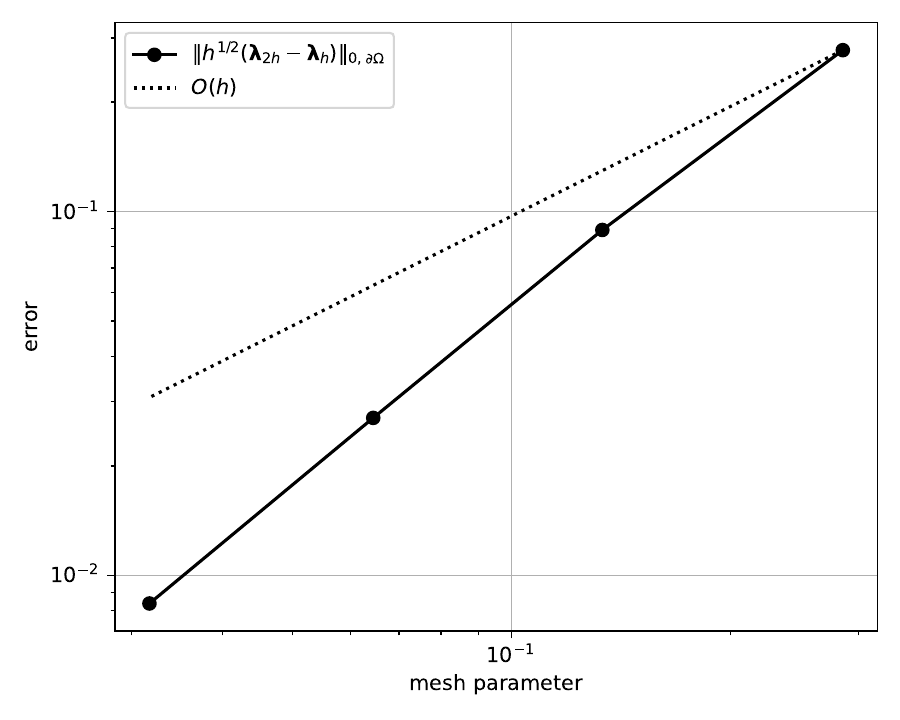}
\caption{The convergence of the error in the Lagrange multiplier.}
\label{fig:lamerr}
\end{figure}

\begin{figure}
\centering
\vcenteredhbox{\includegraphics[width=0.7\textwidth]{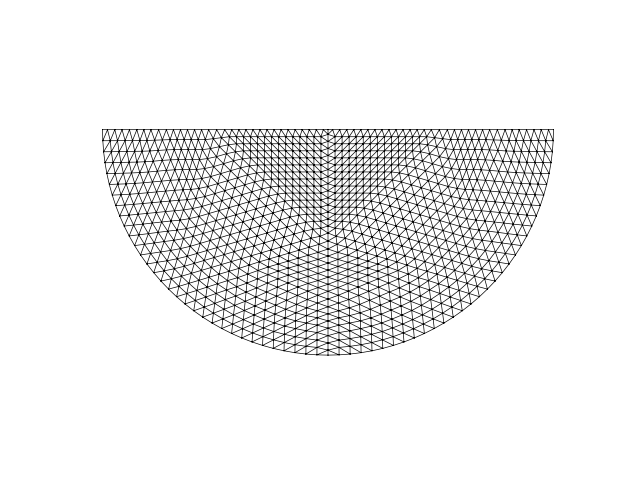}}
\vcenteredhbox{\includegraphics[width=0.70\textwidth]{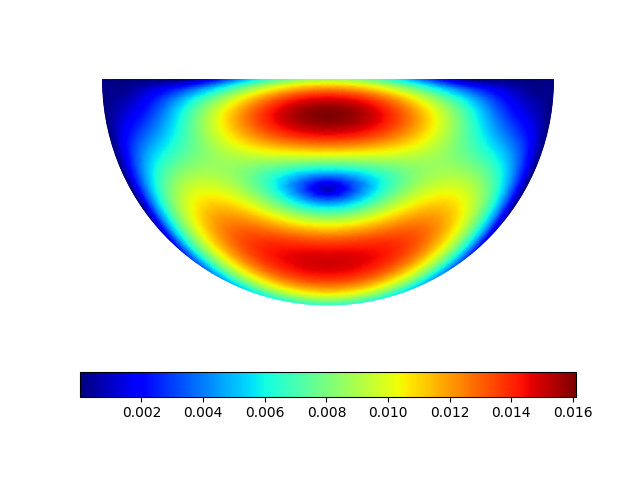}}
\vcenteredhbox{\includegraphics[width=0.70\textwidth]{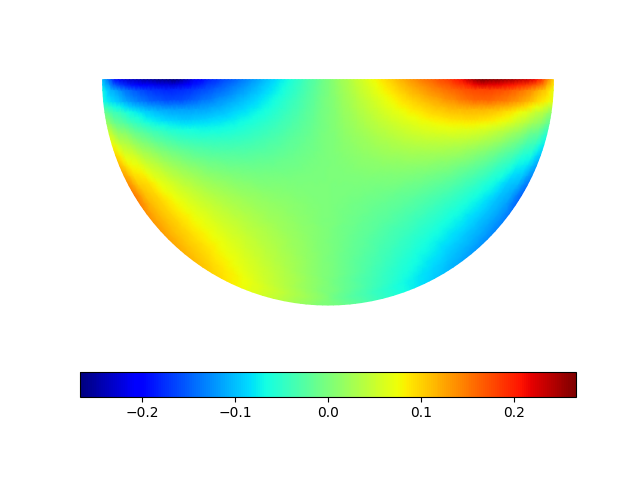}}
\caption{The mesh (top), the velocity magnitude (middle)
and the pressure (bottom)
for the curved boundary experiment.}
\label{fig:discsolsupcircle}
\end{figure}

\begin{figure}
\centering
\vcenteredhbox{\includegraphics[width=0.80\textwidth]{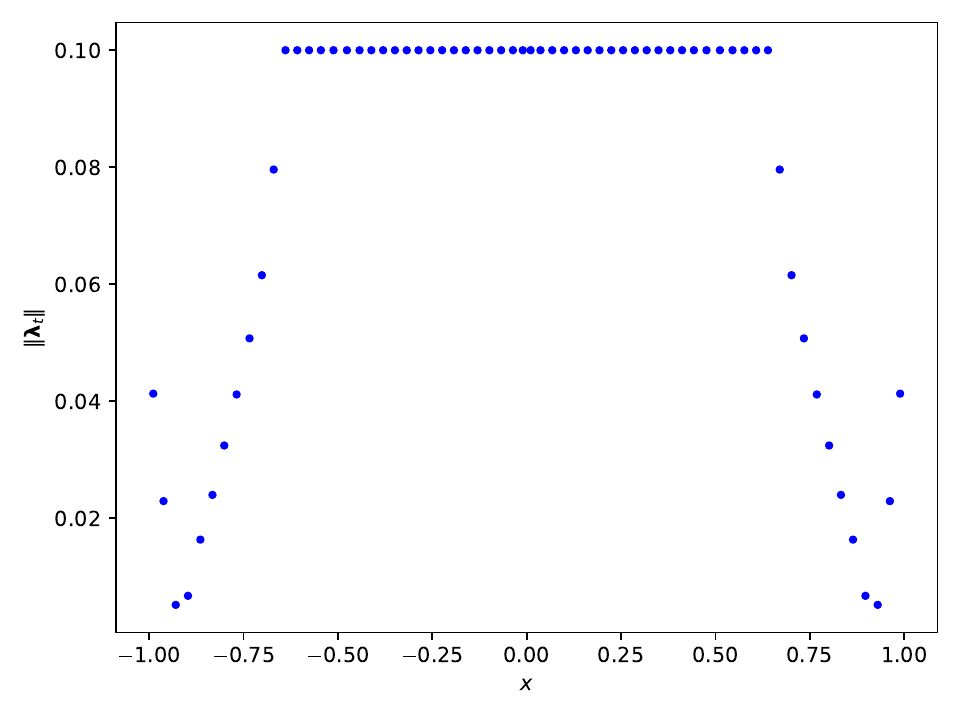}}
\vcenteredhbox{\includegraphics[width=0.80\textwidth]{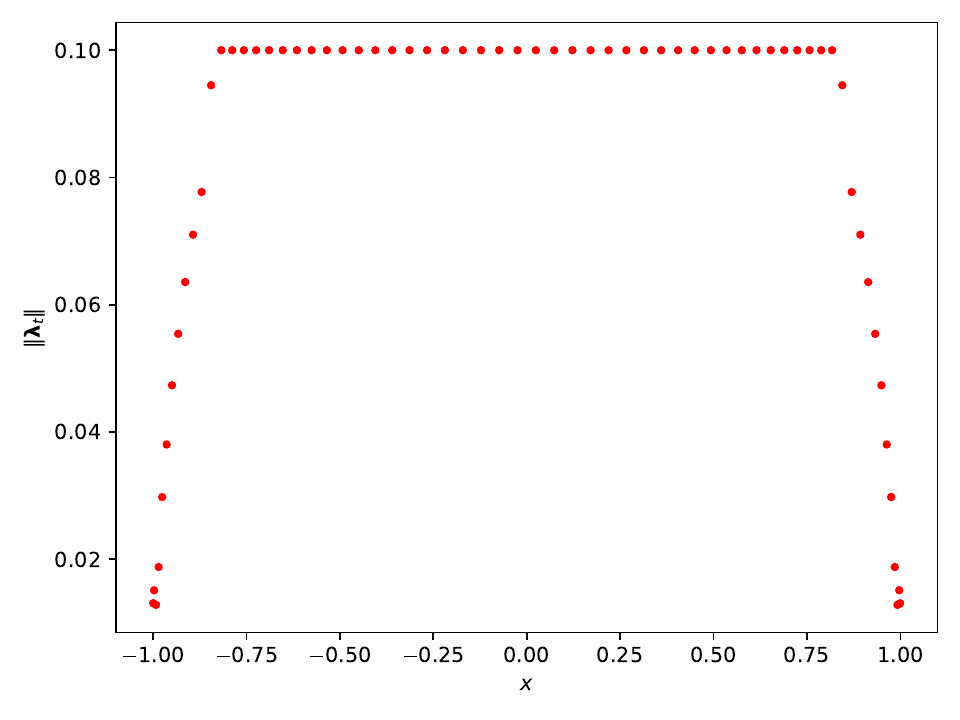}}
\caption{The magnitude of the tangential Lagrange multiplier on the top and the bottom boundaries for the curved boundary experiment.  Note that the Lagrange multiplier is elementwise constant and here the values are sampled at the element midpoints.}
\label{fig:lagmultscircle}
\end{figure}

\section{Conclusions}

We have introduced a stabilised finite element method for the mixed approximation of the Stokes problem with a nonlinear slip boundary condition of friction type.
We have proven stability and provided an a priori error estimate for the method. We have also suggested a solution algorithm and validated it through numerical experiments using the lowest order variant.

\section*{Acknowledgements}

This work was supported by the Academy of Finland (Decision 338341) and by the Portuguese government through FCT (Funda\c{c}\~ao para a Ci\^encia e a Tecnologia), I.P., under the project UIDB/04459/2020.

\section*{Conflicts of Interest}

The authors have no conflicts of interest.

\bibliographystyle{siamplain} 
\bibliography{stokes}

\end{document}